\documentclass{amsart}

\newtheorem{Theorem}{\bf Theorem}
\newtheorem{lemma}[Theorem]{\bf Lemma}
\newtheorem{proposition}[Theorem]{\bf Proposition}

\newtheorem{theorem}[Theorem]{\bf Theorem}
\newcommand{\spmat}[1]{%
  \left[\begin{matrix}#1\end{matrix}\right]%
}

\usepackage[all,cmtip]{xy}

     \title[Indecomposable integrally closed modules of rank 2]{Note on indecomposable integrally closed modules of rank 2 over two-dimensional regular local rings}
     \author{Futoshi Hayasaka}
     \address{Department of Environmental and Mathematical Sciences, Okayama University, 3-1-1
Tsushimanaka, Kita-ku, Okayama, 700-8530, Japan}
     \email{hayasaka@okayama-u.ac.jp}
     \thanks{The first named author was partially supported by JSPS KAKENHI Grant Number JP20K03535.}

     \author{Vijay Kodiyalam}
     \address{The Institute of Mathematical Sciences, Chennai, India and Homi Bhabha National Institute, Mumbai, India}
     \email{vijay@imsc.res.in}
     \thanks{The second named author would like to thank Srikanth Iyengar and Radha Mohan for reigniting his interest in integrally closed modules.}


\begin{document}
     \begin{abstract}
     We characterise ideals in two-dimensional regular local rings that arise as ideals of maximal minors of indecomposable integrally closed modules of rank two.
     \end{abstract}
     \maketitle

    The goal of this short note is to give a definitive answer to the question as to what ideals can arise as the ideal of maximal minors of an indecomposable integrally closed module of rank two over a two-dimensional regular local ring. A related question was first raised in \cite{Kdy1995}. For monomial ideals, a large class of indecomposable integrally closed module of rank two was constructed in \cite{Hys2020} whose associated ideals of maximal minors included non-simple integrally closed ideals. This work was extended in \cite{Hys2021} to integrally closed modules of larger rank.
Both these recent papers deal with monomial ideals and with existence results. The main purpose of this note is to remove this restriction on monomiality and to explore non-existence results. We however restrict attention only to modules of rank two where we have a satisfying characterisation and plan to explore the higher rank case in a future publication.

Throughout this note $(R,{\mathfrak m})$ will be a two-dimensional regular local ring with infinite residue field and $M$ will be a finitely-generated torsion-free $R$-module. We will assume familiarity with the basic results on finitely generated integrally closed $R$-modules from \cite{Kdy1995}, some of which we now recall. If $M$ is a finitely-generated torsion-free $R$-module, then $F = M^{**}$ is a free $R$-module and we will always regard $M$ as a submodule of $F$ generated by the columns of a suitable matrix. The ideal of maximal 
minors of this matrix will be denoted by $I(M)$ and is an ${\mathfrak m}$-primary 
ideal of $R$ (or the whole of $R$ if $M$ is itself free). Recall that the order of an ideal $I$ of $R$ is the largest non-negative integer $k$ such that $I \subseteq {\mathfrak m}^k$. We define  $ord(M)$ to be $ord(I(M))$.
Integrally closed modules satisfy the numerical equation $\mu(M) = ord(M) + rk(M)$ with $\mu(\cdot)$ and $rk(\cdot)$ denoting the minimal number of generators and rank respectively. We will also need to use the classical factorisation theorem of Zariski for integrally closed ideals in two-dimensional regular local rings - for a modern treatment of this, see
\cite{Lpm1988}.

Observe that when $M$ is indecomposable (as a direct sum of modules), it certainly has no free summand and is consequently contained in ${\mathfrak m}F$. Equivalently, in matrix terms, all entries of a matrix whose columns generate $M$ lie in ${\mathfrak m}$, leading to the first restriction on $I(M)$ for such a module, namely, $ord(I(M)) \geq 2$. This condition is definitely not sufficient for an integrally closed ideal to be the ideal of minors of an indecomposable integrally closed $R$-module of rank two, as our first lemma shows.

\begin{lemma}\label{msquare} The ideal ${\mathfrak m}^2$ is not $I(M)$ for any indecomposable integrally closed $R$-module $M$ of rank two.
\end{lemma}

Before we begin the proof, we draw the reader's attention to the determinantal criterion for integral dependence which we will use several times in this paper. The version that we use is the following. With $M \subseteq F$ as before, an element $v \in F$ (thought of as a column vector) is in the integral closure of the module $M$ iff the determinant of any matrix with one column $v$ and the rest of the columns in $M$ is integral over $I(M)$. This follows from the more general statement of Rees in \cite{Res1987}.

\begin{proof}[Proof of Lemma \ref{msquare}]
Suppose $M$ is integrally closed of rank two without free direct summands and that $I(M) = {\mathfrak m}^2$. With $F = M^{**}$, as usual, $F \cong R^2$ and $M \subseteq {\mathfrak m}F$. Since all entries of a matrix for $M$ lie in ${\mathfrak m}$, the determinantal criterion shows that for any $x \in{\mathfrak m}$, the vectors
$$
\left[
\begin{array}{c}
x\\
0
\end{array}
\right]
{\text {~and~}}
\left[
\begin{array}{c}
0\\
x
\end{array}
\right]
$$
are integral over $M$ and hence are in $M$. Thus $M \supseteq {\mathfrak m}F$ as well and so $M = {\mathfrak m}F = {\mathfrak m} \oplus {\mathfrak m}$ and therefore decomposable.
\end{proof}

While $ord(I) \geq 2$ is not sufficient for an integrally closed ideal to be the ideal of minors of an indecomposable 
integrally closed $R$-module of rank two, it turns out that $ord(I) \geq 3$ is such a sufficient condition. Indeed, in the monomial case, this is one of the main results of \cite{Hys2020}.

\begin{proposition}\label{order3}
Let $I$ be an integrally closed ${\mathfrak m}$-primary ideal of $R$ with $ord(I) \geq 3$. Then, there exists an indecomposable 
integrally closed $R$-module $M$ of rank two with $I(M) = I$.
\end{proposition}

The proof of this proposition needs a consequence of the equational criterion for integral dependence. Recall that, in this context, the equational criterion - see \cite{Res1987} - states the following. For $M \subseteq F$ as usual, regard elements of $F$ as linear forms in $r$-variables $X_1,\cdots,X_r$ where $r = rk(F)$.
Let $S$ be the $R$-subalgebra of $R[X_1,\cdots,X_r]$ generated by the elements of $M$ (regarded as linear forms).
An element $F$ is integral over $M$ iff the corresponding linear form is integral over $S$. The consequence we will need is that if an element  $v$ of $F$ is integral over $M$, which is generated by the columns of some matrix, then for each $k$, the $k^{th}$ entry $v_k$ of $v$ is integral over the ideal generated by the $k^{th}$ row of the matrix. This follows by taking the equation for integral dependence of the corresponding linear form over $S$ and setting all variables except $X_k$ to $0$.

We will also need a couple of preparatory lemmas. The ideals $I_1(M)$ and $I_2(M)$ in the next two lemmas refer to the ideals generated by the entries and $2 \times 2$ minors of a matrix whose columns generate $M$. Thus for a module of rank two, $I_2(M) = I(M)$.

\begin{lemma}\label{order2}
Let $I$ be an integrally closed ${\mathfrak m}$-primary ideal of $R$ with $ord(I) \geq 2$. There exists an integrally closed $R$-module $M$ of rank two such that $I_2(M) = I$ and $I_1(M) = {\mathfrak m}$.
\end{lemma}

\begin{proof}
Let $a,b$ form a minimal reduction of $I$ and choose $c \in I \setminus (a,b)$ so that ${\mathfrak m}c \subseteq (a,b)$. This can be done since $I$ strictly contains $(a,b)$ and choosing a filtration of the quotient by copies of the residue field of $R$. Then the ideal $(a,b,c)$ is minimally 3-generated, since its colength in $R$ is strictly smaller than that of $(a,b)$, which is the multiplicity of $I$.

Consider a minimal free resolution of $(a,b,c)$ of the form
\begin{equation*}
\xymatrix@C+2pc{
0 \ar[r]
&
  R^2 \ar[r]^{\spmat{p & q \\ r & s \\ t & u}}
  &
  R^3 \ar[r]^{\spmat{a & b & c}}
  &
  (a,b,c) \ar[r]
  & 
  0.
}
\end{equation*}
Since ${\mathfrak m}c \subseteq (a,b)$, the ideal $(t,u) = {\mathfrak m}$. Let $M$ be the integral closure of the  module that is generated by the columns of the transposed presentation matrix
$$
\left[
\begin{array}{ccc}
p & r & t\\
q & s & u
\end{array}
\right],
$$
so that $M$ is an integrally closed module with ideal of maximal minors $\overline{(a,b,c)} = I$.
Further note that $I_1(M) = {\mathfrak m}$ since $(t,u) = {\mathfrak m}$.
\end{proof}

\begin{lemma}\label{decom} Let $M$ be a decomposable integrally closed module of rank two with $I_2(M) = I$ of order at least 3  and $I_1(M) = {\mathfrak m}$. Then
$M$ decomposes as ${\mathfrak m} \oplus J$ with $J = (I:{\mathfrak m})$ and $I = {\mathfrak m}J$.
\end{lemma}

\begin{proof} Suppose that $M$  decomposes as $K \oplus J$, where we may assume without loss of generality that $K$ and $J$
are ${\mathfrak m}$-primary ideals of $R$. Since $M$ is integrally closed, so are $K$ and $J$. Now $I_2(M) = KJ$ and $I_1(M) = K+J$ and so we have that $KJ = I$ and $K+J = {\mathfrak m}$. Since $I$ is of order at least 3, at least one of $K$ and $J$, say $J$,  must have order at least 2 and then, from $K+J = {\mathfrak m}$ we may conclude that $K$ must be ${\mathfrak m}$. So $M$ must decompose as ${\mathfrak m} \oplus J$ and so $I_2(M) = I = {\mathfrak m}J$.
Now, $J = (I:{\mathfrak m})$ necessarily. The containment $J \subseteq (I:{\mathfrak m})$ is clear and the other direction follows from the familiar ``determinant trick" since $J$ is integrally closed. Finally, $I = I_2({\mathfrak m} \oplus J) = {\mathfrak m}J$.
\end{proof}

\begin{proof}[Proof of Proposition \ref{order3}] 
By Lemma \ref{order2}, there is an integrally closed $R$-module $M$ of rank two with $I_2(M) = I$ and $I_1(M) = {\mathfrak m}$. If $M$ is indecomposable, we're done. Else, by Lemma \ref{decom}, $I = {\mathfrak m}J$ with
$J = (I : {\mathfrak m})$.
Now let $M^\prime$ be the module generated by the columns of the matrix
$$
\left[
\begin{array}{ccc}
J & y & 0\\
0 & x & J
\end{array}
\right],
$$
where $x,y$ are sufficiently general in the sense that $I$ is contracted from $R[\frac{{\mathfrak m}}{y}]$ and $(x,y) = {\mathfrak m}$. This contraction condition is equivalent to $(I:y) = (I:{\mathfrak m})$.

We see that $I(M^\prime) = {\mathfrak m}J + J^2 = I$ since $J^2 \subseteq {\mathfrak m}J = I$. Further, $M^\prime$ is integrally closed. This is an application of the equational criterion. Suppose that
$$
\left[
\begin{array}{ccc}
p\\
q
\end{array}
\right]
$$
is integral over $M^\prime$. Then $p$ is integral over $(y,J)$ while $q$ is integral over $(x,J)$. Since these are both ideals of order 1, they are integrally closed and so $p \in (y,J)$ and $q \in (x,J)$.
Thus, by subtracting suitable elements of $J$ from both $p$ and $q$, we may assume that this vector is of the form
$$
\left[
\begin{array}{ccc}
uy\\
vx
\end{array}
\right],
$$
for some $u,v \in R$.
Now, by the determinantal criterion $(u-v)xy \in I$ and so $(u-v)x \in (I:y) = (I:{\mathfrak m}) = J$. It follows that
$$
\left[
\begin{array}{ccc}
uy\\
vx
\end{array}
\right] = \left[
\begin{array}{ccc}
0\\
(v-u)x
\end{array}
\right] + u\left[
\begin{array}{ccc}
y\\
x
\end{array}
\right] \in M^\prime,
$$
as needed.

If $M^\prime$ is decomposable, again by Lemma \ref{decom}, it decomposes as ${\mathfrak m} \oplus J$. In particular, there is a homomorphism from $M^\prime$ to $R$  - namely projection to the first component - whose image is ${\mathfrak m}$. Regarding $M^\prime$ as a submodule of $F= (M^\prime)^{**}$, any homomorphism from $M^\prime$ to $R$ extends to one from $F$ to $R$ and so is of the form of left multiplication by a row vector, say $[s ~ t]$. Under
this homomorphism the image of $M^\prime$ is the ideal $sJ + (sy+tx) + tJ \subseteq (sy+tx,{\mathfrak m}^2)$. This ideal can clearly never be ${\mathfrak m}$, so $M^\prime$ must be indecomposable.
\end{proof}

Before moving on to non-existence results, we show one more sufficiency condition.

\begin{proposition}\label{otherorder2} Suppose that $I$ is an integrally closed ${\mathfrak m}$-primary ideal that is of order 2 and is either a simple integrally closed ideal or is a product of two simple integrally closed ${\mathfrak m}$-primary ideals, say $I=JK$ where $J+K \neq {\mathfrak m}$.
Then $I$ is the ideal of minors of an indecomposable integrally closed $R$-module of rank two.
\end{proposition}

\begin{proof}
Consider the module $M$ constructed as in the proof of Lemma \ref{order2}. This is integrally closed of rank two with ideal of minors $I$ and $I_1(M) = {\mathfrak m}$. If $M$ decomposes as $P \oplus Q$ with $P$ and $Q$ being ${\mathfrak m}$-primary ideals that are necessarily integrally closed, then $I = PQ$ and $P+Q = {\mathfrak m}$. By Zariski's unique factorisation theorem, $P$ and $Q$ are necessarily $J$ and $K$ up to transposition. But this contradicts $J+K \neq {\mathfrak m}$. Hence $M$ is indecomposable.
\end{proof}

We will next prove the primary non-existence result of this note. This answers a question raised in \cite{Hys2020}.

\begin{proposition}\label{nonexis}
Let ${\mathfrak m} = (x,y)$ and suppose that $I = (x^m,xy,y^n)$ for $m,n \geq 2$. If $M$ is an integrally closed module of rank two without a free direct summand with $I(M) = I$, then $M$ is isomorphic to $(x,y^{n-1}) \oplus (x^{m-1},y)$.
In particular, $M$ is decomposable.
\end{proposition}

\begin{proof} If $m=n=2$, then Lemma \ref{msquare} finishes the proof so we may assume that at least one of $m,n$ is at least 3, and then by symmetry that $m \geq 3$.
Since $M$ has no free direct summand, if $F = M^{**}$, then $M \subseteq {\mathfrak m}F$.
Suppose that $M$ is minimally generated by the columns of the matrix
$$
A=\left[
\begin{array}{rlcc}
p_0x+p_1y & q_0x+q_1y & r_0x+r_1y & s_0x+s_1y\\
t_0x+t_1y & u_0x+u_1y & v_0x+v_1y & w_0x+w_1y
\end{array}
\right].
$$
We may freely perform row and column operations on this matrix and the results all
give modules isomorphic to $M$.

We begin by claiming that all $2 \times 2$ minors of the matrix 
$$
B=\left[
\begin{array}{rlcc}
p_1 & q_1 & r_1 & s_1\\
t_1 & u_1 & v_1 & w_1
\end{array}
\right]
$$
are contained in $(I:y^2) = (x,y^{n-2})$. To see this for the minor $p_1u_1 - q_1t_1$, note that the determinant of the first two columns of the original matrix is $(p_0u_0 - q_0t_0)x^2 + (p_0u_1+p_1u_0-q_0t_1-q_1t_0)xy + (p_1u_1 - q_1t_1)y^2$. Since this is contained in $I$, it follows that $(p_1u_1 - q_1t_1)y^2 \in (I,x) = (x,y^n)$ and so $p_1u_1 - q_1t_1 \in (x,y^n):y^2 = (x,y^{n-2}) = (I:y^2)$. A similar proof holds for the other minors. By the determinantal criterion, every column of $yB$ is in $M$.

Since $xy \in I$, the coefficient of $xy$ in some $2 \times 2$ minor of $A$ must be a unit. Hence some entry of $B$ is
a unit and the corresponding column of $B$  is the second column of a matrix, say $P$, in $GL(2,R)$. If $Q = P^{-1}$, then
every column of $yQB$ is in the module generated by the columns of $QA$. By choice of $Q$, one of the columns of 
$yQB$ is the vector
$$
\left[
\begin{array}{c}
0\\
y
\end{array}
\right].
$$
Hence we may assume that the vector above is in $M$.

Also, if $J = (I:{\mathfrak m}) = (x^{m-1},xy,y^{n-1})$ then by the determinantal criterion again, $JF \subseteq M$. Thus $M$ is generated by the columns of
$$
\left[
\begin{array}{ccccccc}
J & 0 & 0 & p_0x+p_1y & q_0x+q_1y & r_0x+r_1y & s_0x+s_1y\\
0 & J & y & t_0x+t_1y & u_0x+u_1y & v_0x+v_1y & w_0x+w_1y
\end{array}
\right],
$$
and hence also by the columns of
$$
\left[
\begin{array}{ccccccc}
J & 0 & 0 & p_0x+p_1y & q_0x+q_1y & r_0x+r_1y & s_0x+s_1y\\
0 & x^{m-1} & y & t_0x & u_0x & v_0x & w_0x
\end{array}
\right].
$$

All the elements of the first row must be in $(I:y) = (x,y^{n-1})$ and since $y^{n-1} \in J$, a matrix of the form
$$
\left[
\begin{array}{ccccccc}
J & 0 & 0 & p_0x & q_0x& r_0x& s_0x\\
0 & x^{m-1} & y & t_0x & u_0x & v_0x & w_0x
\end{array}
\right].
$$
has columns that generate $M$.

Since $xy \in I$ (and $x \notin J$ because $m \geq 3$) at least one of $p_0,q_0,r_0,s_0$ must be a unit. Assume without loss of generality that $p_0$ is a unit and then by column operations (and renaming) we may assume that $q_0,r_0,s_0$ vanish.  So $u_0,v_0,w_0 \in (I:x^2) = (x^{m-2},y)$ and so by more column operations, we may assume that $u_0x,v_0x,w_0x$ also all vanish. So $M$ is generated by the columns of a matrix of the form
$$
\left[
\begin{array}{ccccccc}
J & 0 & 0 & p_0x\\
0 & x^{m-1} & y & t_0x
\end{array}
\right].
$$
Finally by more row and column operations $M$ is generated by columns of a matrix of the form
$$
\left[
\begin{array}{ccccccc}
y^{n-1} & 0 & 0 & x\\
0 & x^{m-1} & y & 0
\end{array}
\right].
$$
and is therefore isomorphic to $(x,y^{n-1}) \oplus (x^{m-1},y)$, as needed.
\end{proof}

We conclude with the statement and proof of our main theorem.

\begin{theorem}
If $(R,{\mathfrak m})$ is a two-dimensional regular local ring with infinite residue field and $I$ is an integrally closed
${\mathfrak m}$-primary ideal of $R$, then $I$ is the ideal of minors of an indecomposable integrally closed $R$-module of rank two exactly when it satisfies one of the following conditions:
\begin{enumerate}
\item $ord(I) \geq 3$.
\item $ord(I) = 2$ and $I$ is a simple integrally closed ideal.
\item $ord(I) = 2$ and $I$ is a product of simple integrally closed ideals $J$ and $K$ with $J+K \neq {\mathfrak m}$.
\end{enumerate}
\end{theorem}

\begin{proof} The sufficiency of the conditions (1), (2), (3) has been established in Proposition \ref{order3} and Proposition \ref{otherorder2}. As for necessity, let $I$ be an integrally closed ${\mathfrak m}$-primary ideal of $R$  satisfying none of (1), (2) and (3). Then either $ord(I) = 1$ - in which case is it clearly impossible for it to be the ideal of minors of an indecomposable integrally closed $R$-module of rank two - or $ord(I) = 2$ and $I$ is a product of simple integrally closed ideals $J$ and $K$ with $J+K = {\mathfrak m}$. In this case, $ord(J) = 1 = ord(K)$ and a little thought shows that $I$ must necessarily be of the form $(x^m,xy,y^n)$ for $m,n \geq 2$ and $x,y,$ being generators of ${\mathfrak m}$. Now appeal to Proposition \ref{nonexis} to see that $I$ is not the ideal of minors of an indecomposable integrally closed $R$-module of rank two.
\end{proof}

\section*{Acknowledgements}
The first named author was partially supported by JSPS KAKENHI Grant Number JP20K03535. The second named author would like to thank Srikanth Iyengar and Radha Mohan for reigniting his interest in integrally closed modules.


\begin{thebibliography}{10}

\bibitem[Hys2020]{Hys2020}
Hayasaka, Futoshi.
\newblock{Constructing indecomposable integrally closed modules over a two-dimensional regular local ring.}
\newblock {\em Journal of Algebra} 556 : 897-907, 2020.

\bibitem[Hys2021]{Hys2021}
Hayasaka, Futoshi.
\newblock{Indecomposable integrally closed modules of arbitrary rank over a two-dimensional regular local ring.}
To appear in \newblock{\em Journal of Pure and Applied Algebra} 226, 2022, 
\newblock {\em arXiv:2112.02885v1}, 2021.
              
\bibitem[Kdy1995]{Kdy1995}
Kodiyalam, Vijay.
\newblock{Integrally closed modules over two-dimensional regular local rings.}
\newblock{\em Transactions of the  American  Mathematical Society} 347(9) : 3551 - 3573, 1995.


\bibitem[Lpm1988]{Lpm1988}
Lipman, Joseph.
\newblock{On complete ideals in regular local rings.}
\newblock{In {\em Algebraic Geometry and Commutative Algebra, Vol. I}}, Kinokuniya, Tokyo :  203 - 231, 1988.

\bibitem[Res1987]{Res1987}
Rees, David.
\newblock{Reduction of modules.}
\newblock{\em Mathematical Proceedings of the  Cambridge Philosophical Society} 101 : 431-449, 1987.



\end{thebibliography}
     \end{document}